\documentclass{dcds-b}
\usepackage{amsmath,amsxtra,amssymb,latexsym,epsfig,amscd,amsthm,subfigure,fancybox,epsfig}
\usepackage[mathscr]{eucal}
\usepackage{graphicx}
\usepackage{epsfig} % for postscript graphics files
\usepackage{epstopdf}
\usepackage{cases}
\usepackage{subfig}
\usepackage{xcolor}

\def\currentvolume{24}
 \def\currentissue{??}
  \def\currentyear{2019}
   \def\currentmonth{??}
    \def\ppages{X--XX}
     \def\DOI{??/dcdsb.??}

\def\para#1{\vskip .4\baselineskip\noindent{\bf #1}}

\newtheorem{thm}{Theorem}[section]

\newtheorem{deff}{Definition}[section]

\newtheorem{lem}{Lemma}[section]
\newtheorem{prop}{Proposition}[section]
\theoremstyle{definition}

\theoremstyle{remark}
\newtheorem{rem}{Remark}

\numberwithin{equation}{section}

\DeclareMathOperator{\ess}{ess}

\newcommand{\eps}{\varepsilon}

\newcommand{\F}{\mathcal{F}}

\newcommand{\E}{\mathbb{E}}

\newcommand{\N}{\mathbb{N}}
\newcommand{\0}{\mathcal{O}}
\newcommand{\PP}{\mathbb{P}}

\newcommand{\R}{\mathbb{R}}

\newcommand{\abs}[1]{\left\vert#1\right\vert}

\numberwithin{equation}{section}

\DeclareMathOperator*{\esssup}{ess\,sup}

\newcommand{\bed}{\begin{displaymath}}
\newcommand{\eed}{\end{displaymath}}
\newcommand{\bea}{\bed\begin{array}{rl}}
\newcommand{\eea}{\end{array}\eed}

\newcommand{\barray}{\begin{array}{ll}}
\newcommand{\earray}{\end{array}}

\def\disp{\displaystyle}

\newcommand{\1}{\boldsymbol{1}}

\def\bar{\overline}
\def\hat{\widehat}
\def\a.s{\text{\;a.s.\;}}

\title[SPDE Models for Predator-Prey Equations]{Stochastic Partial Differential Equation Models for Spatially Dependent Predator-Prey Equations}

\author[Nhu N. Nguyen, George Yin]{}

\subjclass{Primary: 60H15,
  92D25, 92D40, 35Q92.}

\keywords{Predator-prey model, SPDE, mild solution, positivity, extinction, permanence.}

\email{nhu.math.2611@gmail.com}
\email{gyin@wayne.edu}

\thanks{This
research was supported in part by the National Science Foundation under grant DMS-1710827.}

\begin{document}

\maketitle

\centerline{\scshape Nhu N. Nguyen}
\medskip
{\footnotesize
 \centerline{Department of Mathematics}
   \centerline{Wayne State University}
   \centerline{Detroit, MI 48202, USA}
}

\medskip

\centerline{\scshape George Yin}
\medskip
{\footnotesize
 % please put the address of the second  and third author
 \centerline{Department of Mathematics}
   \centerline{Wayne State University}
   \centerline{Detroit, MI 48202, USA}
}

\begin{abstract}
Stemming from the stochastic Lotka-Volterra or
predator-prey equations, this work aims to model the spatial inhomogeneity by using stochastic partial differential equations (SPDEs).
Compared to the
classical models, the SPDE models are more versatile.
To incorporate
 more qualitative features of the ratio-dependent models,
 the
 Beddington-DeAngelis functional
response is also used. To analyze the systems under consideration, first existence and uniqueness of solutions of the SPDEs are obtained using the notion of mild solutions.  Then sufficient conditions for permanence and extinction are derived.

\end{abstract}

\section{Introduction}
The predator-prey models or  Lotka-Volterra equations
have a long history and have been widely studied
because of their importance in ecology.
Such models have also been used in for example, statistical mechanics and other related fields.
In 1925, the model was first introduced in \cite{Lotka}
as
follows
\begin{equation*}
\begin{cases}
\disp
{dU(t) \over dt} =\big[U(t)\big(a-bV(t)\big)\big],\\[2ex]
\disp {dV(t) \over dt}=\big[V(t)\big(-c+fU(t)\big)\big].
\end{cases}
\end{equation*}
To improve the model,
 the prey and predator self-competition terms have been added to the original model while
different types of functional responses such as  Holling types I-III \cite{Holling}, ratio-dependence type \cite{Arditi}, and Beddington-DeAngelis type \cite{Bed, DeAn}, etc.,
have also been considered.
Recently, Li et al.   studied a predator-prey system with Beddington-DeAngelis functional response in \cite{Li-11}, in which the density functions are spatially homogeneous. The model is represented by 
\begin{equation*}
\begin{cases}
\disp
{dU(t) \over dt}=\Big[U(t)\big(a_1-b_1U(t)\big)-\dfrac{c_1U(t)V(t)}
{m_1+m_2U(t)+m_3V(t)}\Big],\;t\geq 0,\\[2ex]
\disp {dV(t) \over dt}=\Big[V(t)\big(-a_2-b_2V(t)\big)
+\dfrac{c_2U(t)V(t)}{m_1+m_2U(t)+m_3V(t)}\Big],\;t\geq 0,
\end{cases}
\end{equation*}
where $a_i$, $b_i$, $c_i$, and $m_i$ are positive constants. Although significant progress has been made, it is well recognized that noise effect often needs to be taken into consideration and that
allowing spacial inhomogeneous variation could improve the model further.
To take  environment noise into consideration,
one considers a stochastic differential equation model as follows
\begin{equation*}
\begin{cases}
dU(t)\!\!=\!\!\Big[U(t)\big(a_1-b_1U(t)\big)-\dfrac{c_1U(t)V(t)}
{m_1+m_2U(t)+m_3V(t)}\Big]dt+\sigma_1U(t)dB_1(t),t\geq 0,\\[1ex]
dV(t)\!\!=\!\!\Big[V(t)\big(-a_2-b_2V(t)\big)
\!+\dfrac{c_2U(t)V(t)}{m_1+m_2U(t)+m_3V(t)}\Big]dt\!+\!\sigma_2V(t)dB_2(t),t\geq 0,
\end{cases}
\end{equation*}
where $B_1(t)$ and $B_2(t)$ are independent and real-valued Brownian motions,  and $\sigma_1$ and $ \sigma_2\neq 0$ are intensities of the noises.
Such a problem has been
studied  in \cite{DDY-16}. In fact, the study is related to what is known as Kolmogorov systems, which has a wide range of applications in ecology \cite{DDNY16},  epidemiology, as well as other fields such as social networks.
The long-time behaviors have been characterized by providing a threshold between extinction and permanence.
To make the model more suitable for a wider class of systems, it is natural
to include spatial dependence. In the deterministic setup, it has been shown that not only is the spatial inhomogeneity mathematically interesting, but also it is crucially important for practical concerns. Taking the
  spatially inhomogeneous case into consideration, a predator-prey reaction-diffusion system takes the form
\begin{equation*}
\begin{cases}
\disp {\partial \over \partial t}U(t,x)=d_1\Delta U(t,x) +U(t,x)\big(a_1(x)-b_1(x)U(t,x)\big)
\\\hspace{4.5cm}-\dfrac{c_1(x)U(t,x)V(t,x)}{m_1(x)+m_2(x)U(t,x)+m_3(x)V(t,x)}\;\text{in}\;\R^+\times\0,\\
\disp {\partial \over \partial t}V(t,x)=d_2\Delta V(t,x) + V(t,x)\big(-a_2(x)-b_2(x)V(t,x)\big)
\\\hspace{4.5cm}+\dfrac{c_2(x)U(t,x)V(t,x)}{m_1(x)+m_2(x)U(t,x)
+m_3(x)V(t,x)}\;\text{in}\;\R^+\times\0,\\
\partial_{\nu}U(t,x)=\partial_{\nu}V(t,x)=0\quad\quad\quad\quad\text{on} \;\;\;\;\R^+\times\partial\0,\\
U(x,0)=U_0(x),V(x,0)=V_0(x)\quad\text{in} \;\;\;\;\0,
\end{cases}
\end{equation*}
where $\Delta $ is the Laplacian with respect to the spatial variable, $\0$ is a bounded smooth domain
of $\R^l$ ($l\geq 1$), $\partial_{\nu}$ denotes the directional derivative with the $\nu$ being the outer normal direction on $\partial \0$, and $d_1$ and $d_2$ are positive constants representing the diffusion rates of the prey and predator population densities, respectively. In contrast to the previous cases,
 in lieu of constant values,
 $a_i(x),$ $b_i(x),$ $c_i(x),$ and  $m_i(x)\in C^2(\bar\0,\R)$ are
 allowed to be positive functions. Recently, spatially heterogeneous  systems have been widely studied;
see \cite{Ai-17,G-07,Li-18,LLL,Lou-17,WZ-18}  and reference therein.
It has been demonstrated that including spatial inhomogeneity has provided better models with high fidelity.
As argued in \cite{NeuP}, a fundamental problem faced by ecologists is that the spatial and temporal scales at which measurements are practical. Much evidence demonstrates
 the importance of  interactions and dispersal, and the importance of including spatial dependence in the formulation. In the aforementioned paper, the authors proposed a
  specific spatially dependent model.

 This work presents our initial effort in treating random environmental noise,
  as well as taking into consideration of spatial inhomogeneity.
 In view of the progress to date, this paper proposes and analyzes a
 predator-prey model under stochastic influence and spatial inhomogeneity. We consider a stochastic partial differential equation model with initial and boundary data as follows
\begin{equation}\label{eq}
\begin{cases}
dU(t,x)=\Big[d_1\Delta U(t,x) +U(t,x)\big(a_1(x)-b_1(x)U(t,x)\big)
\\
\hspace{1cm}-\dfrac{c_1(x)U(t,x)V(t,x)}{m_1(x)+m_2(x)U(t,x)+m_3(x)V(t,x)}\Big]dt+ U(t,x)dW_1(t,x), \ \;\text{in}\;\R^+\times\0, \\
dV(t,x)=\Big[d_2\Delta V(t,x) + V(t,x)\big(-a_2(x)-b_2(x)U(t,x)\big)\\
\hspace{1cm}+\dfrac{c_2(x)U(t,x)V(t,x)}{m_1(x)+m_2(x)U(t,x)+m_3(x)V(t,x)}\Big]dt+ V(t,x)dW_2(t,x), \ \;\text{in}\;\R^+\times\0, \\
\partial_{\nu}U(t,x)=\partial_{\nu}V(t,x)=0, \quad\text{ on} \;\;\;\;\R^+\times\partial\0,\\
U(x,0)=U_0(x),V(x,0)=V_0(x), \quad\text{ for } \ x \in\0,
\end{cases}
\end{equation}
where $W_1(t,x)$ and $W_2(t,x)$ are $L^2(\0,\R)$-valued Wiener processes, which represent the noises in both time and space.
We refer the readers to \cite{Prato} for more details on the $L^2(\0,\R)$-valued Winner process.

The rest of the paper is arranged as follows.
Section \ref{sec:pre} gives some preliminary results and also formulates the problem to be studied precisely.
Section \ref{sec:exist} establishes the existence and uniqueness of the solution of the associated stochastic partial differential equations as well as its positivity and its continuous dependence on initial data. Section \ref{sec:longtime} introduces a sufficient condition for the extinction and permanence. Finally, Section \ref{example} provides an example.

\section{Formulation and Preliminaries}\label{sec:pre}
Let $\0$ be a bounded domain in $\R^l$ ($l\geq 1$) having a regular boundary and $L^2(\0,\R)$ be the separable Hilbert spaces, endowed with the scalar product
$$\langle u_1,v_1 \rangle_{L^2(\0,\R)}:=\int_{\0}u_1(x)v_1(x)dx,$$
with the corresponding norm $\sqrt{\langle \cdot,\cdot \rangle}$. We say $u_1\geq 0$ if $u_1(x)\geq 0$ almost everywhere in $\0$.
Moreover, we denote by $L^2(\0,\R^2)$ the space of all functions $u(x)=\big(u_1(x),u_2(x)\big)$ where $ u_1,u_2\in L^2(\0,\R)$ endowed with the inner product
\begin{align*}
\langle u,v \rangle_{L^2(\0,\R^2)}&:=\int_\0 \big\langle u(x),v(x) \big\rangle_{\R^2}dx=\int_{\0}\big(u_1(x)v_1(x)+u_2(x)v_2(x)\big)dx
\\&=\langle u_1,v_1 \rangle_{L^2(\0,\R)}+\langle u_2,v_2 \rangle_{L^2(\0,\R)},
\end{align*}
where
 $u(x)=(u_1(x),u_2(x))$ and
$v(x)=(v_1(x),v_2(x)).$ Then $L^2(\0,\R^2)$ is also a separable Hilbert spaces. In addition, for $\eps>0,p\geq 1$, denote by $W^{\eps,p}(\0,\R^2)$ (as well as $W^{\eps,p}(\0,\R)$) the Sobolev-Slobodeckij space (the Sobolev space with non-integer exponent).

Let $\big\{\Omega, \F,\{\F_t\}_{t\geq 0},\PP\big\}$ be a complete filtered probability space,  $L^p(\Omega;C([0,t]$, $L^2(\0,\R^2)))$ be the space of  predictable processes $u$
that takes
%, which take 
values in $C([0,t]$, $L^2(\0,\R^2))$, $\PP$-a.s. with the norm
$$\abs{u}^p_{L_{t,p}}:=\E \sup_{s\in [0,t]}\abs{u(s)}^p_{L^2(\0,\R^2)}.$$
Assume that $\{B_{k,1}(t)\}_{k=1}^\infty$ and $\{B_{k,2}(t)\}_{k=1}^\infty$
are independent sequences of $\{\F_t\}_{t\geq 0}$-adapted one-dimensional Wiener processes. Fix an orthonormal basis $\{e_k(x)\}_{k=1}^{\infty}$ in $L^2(\0,\R)$ and assume that it is uniformly bounded in $L^{\infty}(\0,\R)$, i.e.,
$$C_0:=\sup_{k\in\N}\esssup_{x\in\0}\abs{e_k(x)}<\infty.$$
We define the infinite dimensional Winner processes $W_i(t)$, the driving noise in equation \eqref{eq} as follows
$$\displaystyle W_i(t)=\sum_{k=1}^{\infty}\sqrt {\lambda_{k,i}}B_{k,i}(t)e_k,\quad i=1,2,$$
where $\{\lambda_{k,i}\}_{k=1}^{\infty}, (i=1,2)$ are sequences of non-negative real numbers satisfying
\begin{equation}\label{nuclearcondition}
\lambda_i :=\sum_{k=1}^{\infty}\lambda_{k,i}<\infty,\quad i=1,2.
\end{equation}
To proceed, let $A_1$ and $A_2$ be Neumann realizations of $d_1\Delta$ and $d_2\Delta$ in $L^2(\0,\R)$, respectively,
where the Laplace operator
is understood in the distribution sense;  see \cite[Appendix A]{Prato}. Then, $A_1$ and $A_2$ are infinitesimal generators of analytic semi-groups $e^{tA_1}$ and $e^{tA_2}$, respectively.
In addition, if we denote $A=(A_1,A_2)$, then it generates an analytic semigroup $e^{tA}=(e^{tA_1},e^{tA_2}).$ In \cite[Theorem 1.4.1]{Davies}, it is proved that the space $L^1(\0,\R^2 )\cap L^\infty(\0,\R^2)$ is
invariant under $e^{tA}$, so that $e^{tA}$ may be extended to a non-negative one-parameter semigroup $e^{tA(p)}$ on $L^p(\0;\R^2 )$, for all $1\leq p\leq\infty$. All these semi-groups are strongly continuous and consistent in the sense that $e^{tA(p)}u=e^{tA(q)}u$ for any $u\in L^p(\0,\R^2)\cap L^q(\0,\R^2)$; see \cite{Cerrai-book}.
Henceforth, we
suppress
the
 parameter $p$ and denote $e^{tA(p)}$ as  $e^{tA}$ whenever there is no confusion.
 Finally, we recall some well-known properties of operators $A_i$ and analytic semi-groups $e^{tA_i}$ for $i=1,2$ as follows
\begin{itemize}
\item $\forall u\in L^2(\0,\R)$ then $\int_ 0^t e^{sA_i}uds\in D(A_i)$ and $A_i(\int_0^t e^{sA_i}uds)=e^{tA_i}u-u.$
\item By Green's identity, it is possible to obtain
$\forall u\in D(A_i)$, $\int_\0A_i u(x)dx=0.$
\item The semigroup $e^{tA}$ satisfies the fowling properties
\begin{equation}\label{eta}
\begin{aligned}
\abs{e^{tA}u}_{L^\infty(\0,\R^2)}\leq c\abs{u}_{L^\infty(\0,\R^2)}\;\text{and}\;\abs{e^{tA}u}_{L^2(\0,\R^2)}\leq c\abs{u}_{L^2(\0,\R^2)}.
\end{aligned}
\end{equation}
\item For any $t,\eps>0$, $p\geq 1$, the semigroup $e^{tA}$ maps $L^p(\0,\R^2)$ into $W^{\eps,p}(\0,\R^2)$
 and $\forall u\in L^p(\0,\R^2)$
\begin{equation}\label{propertyA}
\abs{e^{tA}u}_{\eps,p}\leq c(t\wedge 1)^{-\eps/2}\abs{u}_{L^p(\0,\R^2)},
\end{equation}
for some constant $c$ independent of $u,t$.
\end{itemize}
For further details,
we refer the reader to the monographs \cite{Arendt,Davies,Ouhabaz} and references therein.

\def\disp{\displaystyle}
We rewrite  \eqref{eq} as a stochastic differential equation in infinite dimension
\begin{equation}\label{eq1}
\begin{cases}
\displaystyle dU(t)\!\!=\!\!\Big[A_1U(t)+U(t)\big(a_1-b_1U(t)\big)\!-\!
\dfrac{c_1U(t)V(t)}{m_1+m_2U(t)+m_3V(t)}\Big]dt\! +\! U(t)dW_1(t),\\
\displaystyle dV(t)\!\!=\!\!\Big[A_2V(t)\!+\! V(t)\big(\!-a_2\! -b_2V(t)\big)+
\!\dfrac{c_2U(t)V(t)}{m_1+m_2U(t)+m_3V(t)}\Big]dt\! +\! V(t)dW_2(t),\\
U(0)=U_0,\; \;V(0)=V_0.
\end{cases}
\end{equation}
As usual, we follow Walsh \cite{Walsh} to say that $(U(t),V(t))$ is a mild solution of \eqref{eq1} if
\begin{equation}\label{so}
\begin{aligned}
\begin{cases}
\displaystyle U(t)=e^{tA_1}U_0+&\!\!\!\disp \int_0^t e^{(t-s)A_1}\Big[U(s)\big(a_1-b_1U(s)\big)
\\&\hspace{3cm}-\dfrac{c_1U(s)V(s)}{m_1+m_2U(s)+m_3V(s)}\Big]ds+W_U(t),\\
\displaystyle V(t)=e^{tA_2}V_0+&\!\!\!\disp \int_0^t e^{(t-s)A_2}\Big[V(s)\big(-a_2-b_2V(s)\big)
\\&\hspace{3cm}+\dfrac{c_2U(s)V(s)}{m_1+m_2U(s)+m_3V(s)}\Big]ds+W_V(t),
\end{cases}
\end{aligned}
\end{equation}
where
$$W_U(t)=\int_0^t e^{(t-s)A_1}U(s)dW_1(s)\quad\text{and}\quad W_V(t)=\int_0^t e^{(t-s)A_2}V(s)dW_2(s),$$
or in the vector form
\begin{equation}\label{so-vec}
\displaystyle Z(t)=e^{tA}Z_0+\int_0^t e^{(t-s)A}F(Z(s))ds+W_Z(t),\quad Z_0=(U_0,V_0),
\end{equation}
where $Z(t)=(U(t),V(t))$, $e^{tA}Z_0:=(e^{tA_1}U_0,e^{tA_2}V_0)$, $W_Z(t)=(W_U(t), W_V(t))$ and $F(Z):=\big(F_1(Z),F_2(Z)\big)$, $e^{(t-s)A}F(Z):=(e^{(t-s)A_1}F_1(Z),e^{(t-s)A_2}F_2(Z))$ where
$$F_1(Z):=U(a_1-b_1U)-\dfrac{c_1UV}{m_1+m_2U+m_3V},$$
$$F_2(Z):=V(-a_2-b_2V)+\dfrac{c_2UV}{m_1+m_2U+m_3V}.$$

\begin{rem}{\rm
The first integrals on the right-hand sides of \eqref{so} are understood as Bochner integrals while $W_U(t), W_V(t)$ are the stochastic integrals (stochastic convolutions); see \cite{Prato}.
Moreover, $U(s)$ and $V(s)$ in the stochastic integrals are understood as multiplication operators.
The calculations involving vectors are understood as in the usual sense.
}
\end{rem}
 
 For many problems in population dynamics or ecology, an important question is whether an individual  will die out in the long time. That is, the consideration of extinction or permanence. Since the mild solution is used, let us modify some definitions in \cite{NY-18} as follows.
\begin{deff}{\rm
A population
with density
$u(t,x)$ is said to be extinct in the mean
if
$$\limsup_{t\to\infty}\E\int_\0 u(t,x)dx=0,$$
and that is said to be permanent in the mean if there exist a positive number $\hat\delta$, is independent of initial conditions of population, such that
$$\limsup_{t\to\infty}\E\int_\0 u^2(t,x)dx\geq\hat\delta.$$
In what follows, for convenience, we often suppress the ``in the mean'' when we refer to extinction and permanence in the mean, because we are mainly working with mild solutions.
}\end{deff}

\section{Existence, Uniqueness and Positivity of the Mild Solution}\label{sec:exist}
Since the coefficients are non-Lipschitz and faster than linear growth, the existence and uniqueness of the mild solutions are not obvious.  Although the existence of the mild solution of reaction-diffusion equations with non-Lipschitz term were treated in \cite{Cerrai-03}, we cannot apply directly the result in this paper since our coefficients do not satisfy the conditions in \cite{Cerrai-03}. However, we can follow the method in \cite{Cerrai-03} by considering the coefficients in each compact set
%ball
so that they are Lipschitz continuous and therefore we will define the solution using these solutions. In what follows, without loss of the generality we can assume $\abs{\0}=1$ for simplicity. Moreover, we also assume that the initial values are non-random.

\begin{thm}\label{existence}
For any initial data $0\leq U_0,V_0 \in L^\infty(\0,\R)$, there exists a unique mild solution $(U(t),V(t))$ of \eqref{eq1} belongs to $L^p(\Omega; C([0,T],L^2(\0,\R^2)))$ for any $T>0, p\geq 1.$ Moreover, the solution is positive, i.e., $U(t),V(t)\geq 0$
for any $t$ and depends continuously on initial data.
\end{thm}

\begin{proof}
In this proof, the letter $c$  denotes positive constants whose values may change in
different occurrences. We will write the dependence of constant on parameters explicitly if it is essential.
First, we rewrite the coefficients by defining
%some functions as followings
$$f_1(x,u,v)=u\big(a_1(x)-b_1(x)u\big)-\frac{c_1(x)uv}{m_1(x)+m_2(x)u+m_3(x)v},$$
$$f_2(x,u,v)=v\big(-a_2(x)-b_2(x)v\big)+\frac{c_2(x)uv}{m_1(x)+m_2(x)u+m_3(x)v},$$
where $f_i: \0 \times \R\times \R \rightarrow \R$.
For each $n\in\N$, we define
\[f_{n,i}:=
\begin{cases}
f_i(x,u,v)\quad \text{if}\quad\abs{(u,v)}_{\R^2}\leq n,\\
f_{i}\Big(x,\dfrac {nu}{\abs{(u,v)}_{\R^2}},\dfrac {nv}{\abs{(u,v)}_{\R^2}}\Big) \quad \text{if}\quad \abs{(u,v)}_{\R^2}> n.\\
\end{cases}
\]
For each $n$,
$f_n(x,\cdot,\cdot)=\big(f_{n,1}(x,\cdot,\cdot),f_{n,2}(x,\cdot,\cdot)\big):\R^2\rightarrow \R^2$ is Lipschitz continuous,
uniformly with respect to $x\in\0 $, so that the composition operator $F_n(z)$ associated to $f_n$ (with $z(x)=(u(x),v(x))$),
$$F_n(z)(x)=\big(F_{n,1}(z)(x),F_{n,2}(z)(x)\big):=
\big(f_{n,1}(x,z(x)),f_{n,2}(x,z(x))\big), x\in\0,$$
is Lipschitz continuous in both $L^2(\0, \R^2)$ and $L^\infty(\0,\R^2)$.

We proceed to consider the following problem
\begin{equation}\label{Z_n}
dZ_n(t)=\big[AZ_n(t)+F_n(Z_n(t))\big]dt+Z_n(t)dW(t), \quad Z_n(0)=(U_0,V_0),
\end{equation}
where $Z_n(t)=\big(U_n(t),V_n(t)\big)$, $AZ_n(t):=\big(A_1U_n(t),A_2V_n(t)\big)$ and
$$Z_n(t)dW(t):=\big(U_n(t)dW_1(t),V_n(t)dW_2(t)\big).$$
Since the coefficient in \eqref{Z_n} is Lipschitz continuous, by contraction mapping argument (see \cite[Proof of Theorem 3.1]{NY-18} or \cite{Prato}), we obtain that the equation \eqref{Z_n} admits a unique mild solution $Z_n(t)=(U_n(t),V_n(t))\in L^p(\Omega; C([0,T_0],L^2(\0,\R^2)))$ for some sufficiently small $T_0$. Therefore, for any finite $T>0$, there is a unique mild solution of \eqref{Z_n} in $L^p(\Omega; C([0,T],L^2(\0,\R^2)))$. To proceed, we will prove the positivity of $U_n(t),V_n(t).$

\begin{lem}\label{positivity}
For any  initial condition $0\leq U_0,V_0 \in L^\infty(\0,\R)$, $U_n(t),V_n(t)\geq 0,\;\forall t\in[0,T].$
\end{lem}

\begin{proof}
Let $(U_n^*(t),V_n^*(t))$ be the mild solution of the equation
\begin{equation}\label{U_n^*,V_n^*}
\begin{cases}
d
U_n^*(t)=\Big[A_1 U_n^*(t) + F_{n,1}\big(U_n^*(t)\vee 0,V_n^*(t)\vee 0\big)\Big]dt+ \big(U_n^*(t)\vee 0\big)dW_1(t),\\[1ex]
d
V_n^*(t)=\Big[A_2 V_n^*(t) + F_{n,2}\big(U_n^*(t)\vee 0,V_n^*(t)\vee 0\big)\Big]dt+\big(V_n^*(t)\vee 0\big)dW_2(t),\\%[1ex]
U_n^*(0)=U_0,V_n^*(0)=V_0.
\end{cases}
\end{equation}
For $i=1,2$, let $\lambda_i\in \rho(A_i)$, the resolvent set of $A_i$ and $R_i(\lambda_i):=\lambda_i R_i(\lambda_i,A_i)$, with $R_i(\lambda_i,A_i)$ being the resolvent of $A_i$. For each small $\eps>0$, $\lambda=(\lambda_1,\lambda_2)\in \rho(A_1)\times\rho(A_2),$ by \cite[Theorem 1.3.6]{Kailiu}, there exists a unique strong solution $U_{n,\lambda,\eps}(t,x)$, $V_{n,\lambda,\eps}(t,x)$ of the equation
\begin{equation}\label{4.2}
\begin{cases}
\displaystyle
d
U_{n,\lambda,\eps}(t)=\Big[A_1U_{n,\lambda,\eps}(t)+R_1(\lambda_1)F_{n,1}\big(\eps\Phi(\eps^{-1} U_{n,\lambda,\eps}(t)),\eps\Phi(\eps^{-1} V_{n,\lambda,\eps}(t))\big)\Big]dt
\\\quad\quad\quad\quad\quad\quad+R_1(\lambda_1)\eps\Phi(\eps^{-1} U_{n,\lambda,\eps}(t))dW_1(t),\\
\displaystyle
d
V_{n,\lambda,\eps}(t)=\Big[A_2 V_{n,\lambda,\eps}(t)+ R_2(\lambda_2)F_{n,2}\big(\eps\Phi(\eps^{-1} U_{n,\lambda,\eps}(t)),\eps\Phi(\eps^{-1} V_{n,\lambda,\eps}(t))\big)\Big]dt
\\ \quad\quad\quad\quad\quad\quad+R_2(\lambda_2)\eps\Phi(\eps^{-1} V_{n,\lambda,\eps}(t))dW_2(t),\\
U_{n,\lambda,\eps}(0)=R_1(\lambda_1)U_0\;,\; V_{n,\lambda,\eps}(0)=R_2(\lambda_2)V_0,
\end{cases}
\end{equation}
where $\Phi:\R\to\R$ is a function satisfying
\begin{equation*}
\begin{cases}
\Phi\in C^2(\R),\; \Phi(\xi)=0\;\text{if}\;\xi<0\;\text{and}\;\Phi(\xi)\geq 0\;\text{if}\;\xi\geq 0,\\
\eps\Phi(\eps^{-1}\xi)\to \xi\vee 0\quad \text{as}\;\eps\to 0.
\end{cases}
\end{equation*}
For example,
\[
\Phi(\xi)=\begin{cases}0 \quad\text{if}\;\xi\leq 0,\\
3\xi^5-8\xi^4+6\xi^3 \quad\text{if}\;0<\xi< 1,\\
\xi \quad\text{if}\;\xi\geq1.
\end{cases}
\]
Combining with the convergence property in \cite[Theorem 1.3.6]{Kailiu}, we have
$$(U_{n,\lambda_k,\eps}(t),V_{n,\lambda_k,\eps}(t))\to (U_n^{*}(t),V_n^{*}(t)) \ \hbox{ in } L^p(\Omega; C([0,T],L^2(\0,\R^2)))$$ for some sequences $\{\lambda_k\}\subset \rho(A_1)\times\rho(A_2)$ and as $\eps\to 0$.

Now, as in \cite{Tessitore1998}, let
\[
\varphi(\xi)=\begin{cases}\xi^2-\dfrac 16\quad\quad\;\;\;\text{if}\;\xi\leq -1,\\
-\dfrac {\xi^4}2-\dfrac{4\xi^3}{3}\quad\text{if}\;-1<\xi< 0,\\
0 \quad\quad\quad\quad\quad\;\text{if}\;\xi\geq 0.
\end{cases}
\]
Then
$\varphi''(\xi)\geq 0\;\forall \xi$
and $\varphi'(\xi)\Phi(\xi)=\varphi''(\xi)\Phi(\xi)=0\;\forall \xi$. Because $R(\lambda_i,A_i)$ is positivity preserving, by virtue of It\^o's Lemma (\cite[Theorem 3.8]{Curtain}), we obtain
\begin{equation*}
\begin{aligned}
\int_\0 \varphi(U_{n,\lambda,\eps}(t,x))dx&=d_1\int_0^t\int_\0 \varphi'(U_{n,\lambda,\eps}(s,x))\Delta U_{n,\lambda,\eps}(s,x)dxds
\\&=-d_1\int_0^t \int_\0 \varphi''(U_{n,\lambda,\eps}(s,x))\abs{\nabla U_{n,\lambda,\eps}(s,x)}^2dxds
\\&\leq 0.
\end{aligned}
\end{equation*}
Since $\varphi(\xi)>0$ for all $\xi<0$, we conclude that $\forall n\in \N,\eps\geq 0, \lambda\in\rho(A_1)\times\rho(A_2)$
and $U_{n,\lambda,\eps}(t)\geq 0$ for all $t\in [0,T]$. Similarly, we obtain the positivity of $V_{n,\lambda,\eps}(t)$. Hence, $U_n^{*}(t),V_n^{*}(t)\geq 0$ for all $t\in [0,T]$ a.s. Since $(U_n^{*}(t),V_n^{*}(t))$ is the solution of \eqref{U_n^*,V_n^*} and is positive, $U_n^*(t)=U_n(t), V_n^*(t)=V_n(t)$. As a consequence, we obtain the positivity of $U_n(t),V_n(t)$.
\end{proof}

 We are in a position to show that the sequence $\{Z_n\}_{n=1}^{\infty}$ is bounded
by the following lemma.

\begin{lem}\label{lembounded}
For all $n\in\N$ then
\begin{equation}\label{bounded}
\E \sup_{s\in [0,t]}\abs{Z_n(s)}_{L^\infty(\0,\R^2)}^p\leq c_p(t)\big(1+\abs{Z_0}_{L^\infty(\0,\R^2)}\big),
\end{equation}
where $c_{p}(t)$ is a positive constant that may depend on $p,t$ but is independent of $n$.
\end{lem}
\begin{proof}
Without loss of the generality, we need only consider $p$ being sufficient large such that
we can choose simultaneously $\beta,\eps>0$ satisfying
\begin{equation*}%\label{epsbeta}
\frac 1p<\beta<\frac 12\quad\text{and}\quad \frac lp<\eps<2\big(\beta-\frac 1p\big).
\end{equation*}
 By the definition of mild solution, we have
\begin{equation*}
\begin{aligned}
U_n(t)(x)&=\left(e^{tA_1}U_0\right)(x)+\left(\int_0^t e^{(t-s)A_1}F_{n,1}(U_n(s),V_n(s))ds\right)(x)+W_{U_n}(t)(x),
\end{aligned}
\end{equation*}
 almost everywhere,
where $W_{U_n}(t)=\int_0^t e^{(t-s)A_1}U_n(s)dW_1(s).$
Thus, since $e^{tA_1}$  is positivity preserving and $U_n(t), V_n(t)$ are positive, by definition of $F_{n,1}$ and \eqref{eta}, we obtain
\begin{equation}\label{U_n^3}
\begin{aligned}
&\abs{U_n(t)}_{L^\infty(\0,\R)}=\ess\sup_{x\in\0} \abs{U_n(t)(x)}=\ess \sup_{x\in\0} U_n(t)(x)
\\&=\ess\sup_{x\in\0}\Big[\left(e^{tA_1}U_0\right)(x)+\left(\int_0^te^{(t-s)A_1}F_{n,1}(U_n(s),V_n(s))ds\right)(x)+W_{U_n}(t)(x)\Big]
\\&\leq \ess\sup_{x\in\0}\Big[\left(e^{tA_1}U_0\right)(x)+\left(\int_0^{t}e^{(t-s)A_1}U_n(s)a_1ds\right)(x)+W_{U_n}(t)(x)\Big]
\\&\leq c(t)\Big(\big|U_0\big|_{L^\infty(\0,\R)}+\int_0^t\Big|U_n(s)\Big|_{L^\infty(\0,\R)}ds
+\Big|W_{U_n}(t)\Big|_{L^\infty(\0,\R)}\Big),
\end{aligned}
\end{equation}
where $c(t)$ is a constant depending only on $t$ and independent of $n$.
By using a factorization argument (see e.g., \cite[Theorem 8.3]{Prato}), we have
\begin{equation*}
W_{U_n}(t)=\dfrac{\sin \pi \beta}{\pi}\int_0^t (t-s)^{\beta-1}e^{(t-s)A_1}Y_{U_n}(s)ds,
\end{equation*}
where
$$Y_{U_n}(s)=\int_0^s (s-r)^{-\beta}e^{(s-r)A_1}U_n(r)dW_1(r).$$
It is easily seen from \eqref{propertyA} and H\"older's inequality that
\begin{equation}\label{lambda}
\begin{aligned}
\abs{W_{U_n}(t)}_{\eps,p}
&\leq c_\beta\int_0^t (t-s)^{\beta-1}\big((t-s)\wedge 1\big)^{-\eps/2}\abs{Y_{U_n}(s)}_{L^p(\0,\R)}ds
\\
& \leq c_{\beta,p}(t)\Big(\int_0^t \big((t-s)\wedge 1\big)^{\frac p{p-1}(\beta-\eps/2-1)}ds\Big)^{\frac{p-1}p}\Big(\int_0^t \abs{Y_{U_n}(s)}_{L^p(\0,\R)}^pds\Big)^{\frac 1p}\\
&\leq c_{\beta,p}(t)\Big(\int_0^t \abs{Y_{U_n}(s)}_{L^p(\0,\R)}^pds\Big)^{\frac 1p},\a.s ,
\end{aligned}
\end{equation}
where $c_{\beta,p}(t)$ is some positive constant, independent of $n$.
On the other hand, for all $s\in [0,t]$, almost every $x\in\0$, we have
\begin{equation*}
Y_{U_n}(s,x)=\int_0^s (s-r)^{-\beta} \sum_{k=1}^\infty \sqrt{\lambda_{k,1}}M_1(s,r,k,x)dB_{k,1}(r),
\end{equation*}
where
$$M_1(s,r,k)=e^{(s-r)A_1}U_n(r)e_k.$$
Therefore, applying the Burkholder inequality, we obtain that for all $s\in [0,t]$, almost every $x\in\0$,
\begin{equation*}
\begin{aligned}
&\!\!\!\E \abs{Y_{U_n}(s,x)}^p\leq c_p\E\Big[\int_0^s(s-r)^{-2\beta}\sum_{k=1}^{\infty}\lambda_{k,1}\abs{M_1(s,r,k,x)}^2dr\Big]^{\frac p2}.
\end{aligned}
\end{equation*}
As a consequence,
\begin{equation}\label{Ybeta}
\begin{aligned}
&\E \int_0^t \abs{Y_{U_n}(s)}_{L^p(\0,\R)}^pds\\
&\quad\leq c_p(t) \int_0^t\E\Big(\int_0^s (s-r)^{-2\beta}\lambda_1\sup_{k\in\N}\abs{M_1(s,r,k)}_{L^{\infty}(\0,\R)}^2dr\Big)^\frac p2ds.
\end{aligned}
\end{equation}
Moreover, since the uniformly boundedness property of $\{e_k\}_{k=1}^\infty$ and \eqref{eta}, we have
\begin{equation}\label{Mr}
\sup_{k\in\N}\abs{M_1(s,r,k)}_{L^\infty(\0,\R)}\leq c\abs{U_n(r)}_{L^\infty(\0,\R)},
\end{equation}
for some constant $c$ independent of $n,s,r,u,v$.
Combining \eqref{Ybeta} and \eqref{Mr} implies that
\begin{equation}\label{boundedYbeta}
\begin{aligned}
\E \int_0^t \abs{Y_{U_n}(s)}_{L^p(\0,\R)}^pds&\leq c_p(t)\int_0^t\E\sup_{r\in [0,s]}\abs{U_n(r)}_{L^\infty(\0,\R)}^p\Big(\int_0^s (s-r)^{-2\beta}dr\Big)^\frac p2ds
%\\&\leq c_{\beta,p}(t)\int_0^t\E\sup_{r\in [0,s]}\abs{u(r)-v(r)}_{C(\bar\0,\R^2)}^pds
\\&\leq c_{\beta,p}(t)\int_0^t\E\sup_{r\in [0,s]}\abs{U_n(r)}_{L^\infty(\0,\R)}^pds.
\end{aligned}
\end{equation}
Since $\eps >l/p$, the Sobolev inequality, \eqref{lambda}, and \eqref{boundedYbeta} imply that
\begin{equation}\label{estimateW}
\begin{aligned}
\E\sup_{s\in[0,t]}\abs{W_{U_n}(s)}_{L^\infty(\0,\R)}^p\leq c_p(t)\int_0^t \E\sup_{r\in[0,s]}\abs{U_n(r)}_{L^\infty(\0,\R)}^pds.
\end{aligned}
\end{equation}
Hence, we obtain form \eqref{U_n^3} and \eqref{estimateW} that
\begin{equation}
\E\sup_{s\in[0,t]}\abs{U_n(s)}^p_{L^\infty(\0,\R)}\leq c_p(t)\Big(\abs{U_0}_{L^\infty(\0,\R)}^p+\int_0^t\E\sup_{r\in [0,s]}\Big|U_n(r)\Big|_{L^\infty(\0,\R)}^pds\Big),
\end{equation}
for some positive constant $c_p(t)$ that is independent of $n$.
Therefore, we obtain from Gronwall's inequality that
$$\E\sup_{s\in[0,t]}\abs{U_n(s)}^p_{L^\infty(\0,\R)}\leq c_{p}(t)\big(1+\abs{U_0}^p_{L^\infty(\0,\R)}\big),$$
for some constant $c_{p}(t)$, is independent of $n$.
Similarly, we have the same estimate for $V_n(t)$.
 Thus the Lemma is proved.
\end{proof}

\para{Completion of the Proof of the Theorem.}
For any $n\in\N$, we define
\begin{equation}\label{tau}
\zeta_n:=\inf\{t\geq 0: \abs{Z_n(t)}_{L^\infty(\0,\R^2)} \geq n\},
\end{equation}
with the usual convention that $\inf \emptyset =\infty$ and define $\zeta=\sup_{n\in\N}\zeta_n.$ Then we have
$$\PP\{\zeta<\infty\}=\lim_{T\to\infty}\PP\{\zeta<T\},$$
and for each $T\geq 0$
$$\PP\{\zeta\leq T\}=\lim_{n\to\infty}\PP\{\zeta_n\leq T\}.$$
For any fixed $n\in\N$ and $T\geq 0,$ it follows from Lemma \ref{lembounded} that
$$
\begin{aligned}
\PP\{\zeta_n\leq T\}=\PP\Big\{\sup_{t\in [0,T]}\abs{Z_n(t)}^p_{L^\infty(\0,\R^2)}\geq n^p\Big\}&\leq \dfrac 1{n^p} \E \sup_{t\in [0,T]}\abs{Z_n(t)}^p_{L^\infty(\0,\R^2)}\\
&\leq \dfrac{c_{p}(T)\big(1+\abs{Z_0}_{L^\infty(\0,\R^2)}\big)}{n^p}.
\end{aligned}
$$
It leads to that $\PP\{\zeta_n\leq T\}$ goes to zero as $n\to\infty$ and we get $\PP\{\zeta=\infty\}=1.$ Hence, for any $t\geq 0$, and $\omega\in \{\zeta=\infty\}$, there exists an $n\in \N$ such that $t\leq \zeta_n(\omega)$.
Thus we can define
$$Z(t)(\omega):=Z_n(t)(\omega).$$
To proceed, we need to show that this definition is well-defined, i.e., for any $t\leq \zeta_n\wedge\zeta_m$ then $Z_n(t)=Z_m(t)$, $\PP$-a.s. For $n<m$ we set $\zeta_{m,n}=\zeta_n\wedge \zeta_m.$ By definition of $F_n, F_m$
$$\text{if}\quad \abs{z}_{L^\infty(\0,\R^2)}\leq n\quad\text{then}\quad F_n(z)(x)=F_m(z)(x)\;\text{almost everywhere in}\;\0.$$
Therefore, we have
\begin{equation*}
\begin{aligned}
&Z_n(t\wedge \zeta_{m,n})-Z_m(t\wedge\zeta_{m,n})
\\&\!=\!\!\!\int_0^{t\wedge \zeta_{m,n}} e^{(t-s)A}\Big(F_n(Z_n(s))-F_m(Z_m(s))\Big)ds+W_{Z_n-Z_m}(t\wedge\zeta_{m,n})
\\&\!=\!\!\!\int_0^t \1_{\{s\leq \zeta_{m,n}\}} e^{(t-s)A}\Big(F_m(Z_n(s\wedge \zeta_{m,n}))-F_m(Z_m(s\wedge\zeta_{m,n}))\Big)ds
+W_{Z_n-Z_m}(t\wedge\zeta_{m,n}),
\end{aligned}
\end{equation*}
where
$$W_{Z_n-Z_m}(t):=\Big(\int_0^{t} e^{(t-s)A_1}(U_n(s)-U_m(s))dW_1(s),\int_0^{t} e^{(t-s)A_2}(V_n(s)-V_m(s))dW_2(s)\Big).$$
Using a similar argument for getting \eqref{estimateW}, we  obtain
\begin{equation}\label{estimateWmn}
\begin{aligned}
\E\sup_{s\in[0,t]}&\abs{W_{Z_n-Z_m}(s\wedge\zeta_{m,n})}^p_{L^\infty(\0,\R^2)}\\
&\leq c_p(t)\int_0^{t}\E\sup_{s'\in[0,s]}\abs{Z_n(s'\wedge\zeta_{m,n})
-Z_m(s'\wedge\zeta_{m,n})}_{L^\infty(\0,\R^2)}^pds.
\end{aligned}
\end{equation}
Therefore, combining with property \eqref{eta}, the Lipschitz continuity of $F_m$, and \eqref{estimateWmn} yields
\begin{equation*}
\begin{aligned}
\E\sup_{s\in [0,t]}&\abs{Z_n(s\wedge \zeta_{m,n})-Z_m(s\wedge\zeta_{m,n})}_{L^\infty(\0,\R^2)}^p
\\&\leq c_{p,m}(t)\int_0^t\E\sup_{s'\in [0,s]}\abs{Z_n(s'\wedge\zeta_{m,n})-Z_m(s'\wedge\zeta_{m,n})}_{L^\infty(\0,\R^2)}^pds.
\end{aligned}
\end{equation*}
Gronwall's inequality implies that $Z_n(t\wedge\zeta_{m,n})=Z_m(t\wedge\zeta_{m,n})\;\forall t$ or
\begin{equation}\label{z-welldefined}
Z_n(t)=Z_m(t),\quad \quad \forall t\leq \zeta_m\wedge\zeta_n.
\end{equation}
It is clear that the process $Z(t)$ defined as above is a mild solution of \eqref{so-vec}. Indeed, for any $t\geq 0$, $\omega \in \{\zeta=\infty\}$ then there exists $n\in\N$ such that $t\leq\zeta_n$ and
\begin{equation*}
\begin{aligned}
Z(t)&=Z_n(t)=e^{tA}Z_0+\int_0^t e^{(t-s)A}F_n(Z_n(s))ds+W_{Z_n}(t)
\\&=e^{tA}Z_0+\int_0^t e^{(t-s)A}F(Z(s))ds+W_Z(t).
\end{aligned}
\end{equation*}
Next, we prove that such solution is unique. If there exists an other solution $\hat Z(t)$ of \eqref{so-vec}. By the argument in the processing of getting \eqref{z-welldefined}, it is possible to obtain
$$Z(t\wedge \zeta_n)=\hat Z(t\wedge \zeta_n),\quad\forall n\in\N,t\geq 0.$$
Since $\zeta_n\to \infty$ as $n\to\infty$ $\PP$-a.s., we get $Z(t)=\hat Z(t).$ Finally, we  show that $Z(t)\in L^p(\Omega; C([0,T],L^2(\0,\R^2)))$. Indeed, for any $p\geq 1, T>0$,
$$
\begin{aligned}
\sup_{t\in[0,T]}\abs{Z(t)}_{L^2(\0,\R^2)}^p&=\lim_{n\to\infty}\sup_{t\in[0,T]}\abs{Z(t)}_{L^2(\0,\R^2)}^p \1_{\{T\leq\zeta_n\}}
\\&=\lim_{n\to\infty}\sup_{t\in[0,T]}\abs{Z_n(t)}_{L^2(\0,\R^2)}^p \1_{\{T\leq\zeta_n\}}.
\end{aligned}
$$
Hence, by the boundedness of $Z_n(t)$ in Lemma \ref{lembounded}, we obtain that the equation \eqref{so} admits a unique mild solution $Z(t)=(U(t),V(t))\in L^p(\Omega; C([0,T],L^2(\0,\R^2)))$. The positivity of $U(t),V(t)$ follow positivity of $U_n(t),V_n(t).$

To complete the proof, we prove that the solution depends continuously on initial data. For  convenience, we use superscripts to indicate the dependence of the solution on initial values. Let $Z^{z_1}(t),Z^{z_2}(t)$ and $Z_n^{z_1}(t),Z_n^{z_2}(t)$ be the solutions of \eqref{so-vec} and \eqref{Z_n} with initial conditions $Z(0)=Z_n(0)=z_1$ and $Z(0)=Z_n(0)=z_2$, respectively. As in the  proof of the first part, since the Lipschitz continuity of $F_n$, it is easy to obtain that
\begin{equation}\label{depend}
\abs{Z_n^{z_1}-Z_n^{z_2}}_{L_{T,p}}^p\leq c_{n,p}(T) \abs{z_1-z_2}_{L^2(\0,\R^2)}^p.
\end{equation}
Consider the stopping times $\zeta_n^{z_1}$ and $\zeta_n^{z_2}$ as in \eqref{tau}, we have
\begin{equation}\label{Zu-Zv}
\begin{aligned}
&\abs{Z^{z_1}-Z^{z_2}}_{L_{T,p}}^p= \E \sup_{s\in[0,T]} \abs{Z^{z_1}(s)-Z^{z_2}(s)}_{L^2(\0,\R^2)}^p\1_{\{\zeta_n^{z_1}\wedge\zeta_n^{z_2}>T\}}
\\&\quad+\E \sup_{s\in[0,T]} \abs{Z^{z_1}(s)-Z^{z_2}(s)}_{L^2(\0,\R^2)}^p\1_{\{\zeta_n^{z_1}\wedge\zeta_n^{z_2}\leq T\}}
\\& \leq \abs{Z_n^{z_1}-Z_n^{z_2}}_{L_{T,p}}^p+ c_p\Big(1+\abs{Z^{z_1}}_{L_{T,2p}}^p+\abs{Z^{z_2}}_{L_{T,2p}}^p\Big)\big(\PP\{\zeta_n^{z_1}\wedge\zeta_n^{z_2}\leq T\}\big)^{1/2}.
\end{aligned}
\end{equation}
Moreover, it follows from \eqref{bounded} that
\begin{equation*}
\begin{aligned}
\PP\{\zeta_n^{z_1}\wedge\zeta_n^{z_2}\leq T\}&\leq \PP\{\sup_{s\in[0,T]}\abs{Z_n^{z_1}(s)}_{L^\infty(\0,\R^2)}\geq n\}\!\!+\!\!\PP\{\sup_{s\in[0,T]}\abs{Z_n^{z_2}(s)}_{L^\infty(\0,\R^2)}\geq n\}
\\&\leq \dfrac {c_4(T)}{n^4}\Big(1+\abs{z_1}^4_{L^\infty(\0,\R^2)}+\abs{z_2}^4_{L^\infty(\0,\R^2)}\Big).
\end{aligned}
\end{equation*}
Therefore, by applying  \eqref{bounded} once more, we obtain from \eqref{Zu-Zv} and \eqref{depend} that
\begin{equation}\label{Z^u-Z^v}
\abs{Z^{z_1}-Z^{z_2}}_{L_{T,p}}^p\leq c_{n,p}(T)\abs{z_1-z_2}_{L^2(\0,\R^2)}^p+\dfrac{c(T)}{n^2}\Big(1+\abs{z_1}_{L^\infty(\0,\R^2)}^{p+2}+\abs{z_2}_{L^\infty(\0,\R^2)}^{p+2}\Big).
\end{equation}
Hence, for any fixed $z_1\in L^\infty(\0,\R^2)$ and $\eps>0$, we first find $\bar n\in\N$ such that
$$\dfrac{c(T)}{\bar n^2}\Big(1+\abs{z_1}_{L^\infty(\0,\R^2)}^{p+2}+\big(1+\abs{z_1}_{L^\infty(\0,\R^2)}\big)^{p+2}\Big)<\dfrac{\eps}2.$$
By determining $0<\delta^*<1$ such that
$$c_{\bar n,p}(T)\abs{z_1-z_2}_{L^2(\0,\R^2)}^p<\dfrac{\eps}2
\quad\text{whenever}\quad\abs{z_1-z_2}_{L^2(\0,\R^2)}<\delta^*,$$
the continuous dependence of the solution on initial values is proved.
\end{proof}

\begin{rem}\label{rm2}{\rm We have following observations
\begin{itemize}
\item[(i)]
As the above proof, we note that the results in Theorem \ref{existence} still hold if we replace the space $L^p(\Omega; C([0,T],L^2(\0,\R^2)))$ by $L^p(\Omega; C([0,T],L^q(\0,\R^2)))$ for $q\geq 2$ (with $q= \infty$ is allowed).
\item[(ii)]
From now, the solution $Z_n(t)$ of \eqref{Z_n} is called as ``truncated solution" of equation \eqref{so-vec}. By the same argument in the processing of obtaining \eqref{Z^u-Z^v}, we conclude that
$$\abs{Z-Z_n}_{L_{t,p}}\leq \dfrac{c_{p,Z_0}(t)}{n^2}\;\text{for some constant}\;c_{p,Z_0}(t)\;\text{being independent of}\;n.$$
As a consequence
$$\lim_{n\to\infty}\abs{Z-Z_n}_{L_{t,p}}=0.$$
\end{itemize}
}
\end{rem}

\section{Sufficient Conditions for Extinction and Permanence}\label{sec:longtime}
In this section, we investigate the
  longtime behavior of  system \eqref{eq1} by providing sufficient conditions for extinction and permanence. Because we can not apply It\^o's formula to the mild solution as usual, it is very difficult to calculate and estimate.  Following our idea in \cite{NY-18}, we approximate the mild solution $(U(t),V(t))$ of \eqref{eq1} by a sequence of strong solutions (see \cite{Prato} for more details about strong solutions, weak solutions, and mild solutions). Consider the following equation
\begin{equation}\label{StrongSolution}
\begin{cases}
d\bar U_n(t,x)=\Big[d_1\Delta \bar U_n(t,x)+\bar U_n(t,x)\Big(a_1(x)-b_1(x)\bar U_n(t,x)\Big)
\\\hspace{5cm}-\dfrac{c_1(x)\bar U_n(t,x)\bar V_n(t,x)}{m_1(x)+m_2(x)\bar U_n(t,x)+m_3(x)\bar V_n(t,x)}\Big]dt
\\\hspace{5cm}+ \displaystyle \sum_{k=1}^n \sqrt{\lambda_{k,1}}e_k(x)\bar U_n(t,x)dB_{k,1}(t)\quad\text{in} \;\;\;\;\R^+\times\0, \\
d\bar V_n(t,x)=\Big[d_2\Delta \bar V_n(t,x)+\bar V_n(t,x)\Big(-a_2(x)-b_2(x)\bar V_n(t,x)\Big)
\\\hspace{5cm}+\dfrac{c_2(x)\bar U_n(t,x)\bar V_n(t,x)}{m_1(x)+m_2(x)\bar U_n(t,x)+m_3(x)\bar V_n(t,x)}\Big]dt
\\\hspace{5cm}+\displaystyle \sum_{k=1}^n \sqrt{\lambda_{k,2}}e_k(x)\bar V_n(t,x)dB_{k,2}(t)\quad\text{in} \;\;\;\;\R^+\times\0,\\
\partial_{\nu}\bar U_n(t,x)=\partial_{\nu}\bar V_n(t,x)=0\quad\quad\quad\quad\quad\;\;\;\;\text{on} \;\;\;\;\R^+\times\partial\0,\\
\bar U_n(x,0)=U_0(x),\bar V_n(x,0)=V_0(x)\quad\quad\quad\;\text{in} \;\;\;\;\0.
\end{cases}
\end{equation}
Denoted by $E$ the Banach space $C(\bar \0,\R^2)$ and by $A_E$ the part of $A=(A_1,A_2)$ in $E$. Since we assumed the domain $\0$ has regular boundary in our boundary condition, $D(A_E)$ is dense in $E$; see \cite[Appendix A.5.2]{Prato}. Moreover, we  use the following notation:
$$\forall u\in C(\bar\0,\R),\quad \abs{u}^*:=\sup_{x\in\bar\0}u(x),\quad \abs{u}_*:=\inf_{x\in\bar \0}u(x),$$
$$\forall u\in L^p(\0,\R),\quad u\geq 0,\quad \abs{u}_p:=\Big(\int_\0 u^p(x)dx\Big)^{1/p},\;p=1,2,\dots$$

\begin{prop}
Assume that for each $k\in \N$,
%$\forall\;1\leq k\in \N,
$e_k\in C^2(\bar\0,\R)$. For any $0\leq (U_0,V_0)\in D(A_E)$, equation \eqref{StrongSolution} has a unique strong solution $\bar Z_n(t)=(\bar U_n(t),\bar V_n(t))$. Moreover, the solution is positive, i.e., $\bar U_n(t), \bar V_n(t)\geq 0$ and for any finite $T>0$, $(\bar U_n(t),\bar V_n(t))\in \L^p(\Omega, C([0,T],E)).$
\end{prop}

\begin{proof}
We apply the results in \cite{Prato85} or \cite[Section 7.4]{Prato} by verifying certain conditions.
%Some conditions are direct verified; some others are modified.
Define the following linear operators in $L^2(\0,\R^2)$
$$B_k(u,v):=\Big( \sqrt{\lambda_{k,1}}e_ku,\sqrt{\lambda_{k,2}}e_kv\Big), D(B_k)=L^2(\0,\R^2),1\leq k\leq n,$$
$$C:=A-\dfrac 12\sum_{k=1}^n B_k^2-1,\;\;D(C)=D(A).$$
First, the operators $B_k$ generate mutually commuting semi-groups and all above operators and their restrictions on $E$ generate strongly continuous and analytic semi-groups; see \cite[Appendix A.5.2]{Prato} or \cite[Chapter 2]{Arendt}. As a result, the conditions $H_1,H_2(a),H_2(b')$ in \cite{Prato85} are satisfied. Moreover, by the arguments in \cite[Example 6.31]{Prato}, we can conclude that the condition $H_2(c)$ in \cite{Prato85} is also satisfied.
Second, it follows from \cite[Proof of theorem 2 and Appendix A]{Prato85} or \cite{Acqui} that we can modify the condition $H_2(e)$ in \cite{Prato85} by an alternative one, namely, $\bar F_E(X^{\theta_1})\subset X^{\theta_2}$ for some $\theta_1,\theta_2\in (0,\frac 12),$ where $X^\theta:=D(-C_E)^{\theta}$ is the domain of the fractional power operator $(-C_E)^\theta$, $(-C_E)$ is the part of $(-C)$ in $E$ and $\bar F_E$ is the part of $\bar F$ in $E$
$$\bar F(\bar U,\bar V)\!\!=\!\!\Big[\bar U(1+a_1-b_1\bar U)\!-\!\dfrac{c_1\bar U\bar V}{m_1+m_2\bar U+m_3\bar V},\bar V(1-a_2-b_2\bar V)+\dfrac{c_2\bar U\bar V}{m_1+m_2\bar U+m_3\bar V}\Big].$$
By \cite[Proposition A.13]{Prato}, we have for all $\theta_1>\theta_2\in (0,1)$
$$D\big((-C_E)^{\theta_1}\big)\subset D_{C_E}(\theta_1,\infty)\subset D\big((-C_E)^{\theta_2}\big),$$
where $D_{C_E}(\theta_1,\infty)$ is defined as in \cite[Appendix A]{Prato} and by \cite[Appendix A.5.2, p. 399]{Prato}
$$D_{C_E}(\theta_1,\infty)=C^{2\theta_1}(\bar\0,\R^2)\;\text{if}\;\theta_1\in (0,\frac 12),$$
where $C^{2\theta_1}(\bar\0,\R^2)$ is a H\"older's space.
Since the space $C^{2\theta_1}(\bar \0,\R^2)$ satisfies that $u,v\in C^{2\theta_1}(\bar\0,\R^2)$ implies $uv\in C^{2\theta_1}(\bar\0,\R^2)$, we obtain that $\bar F_E(X^{\theta_1})\subset X^{\theta_2}$ for some $\theta_2<\theta_1\in (0,\frac 12)$. Finally, it is needed to verify the monotonicity type hypothesis $H_2(d')$, namely, there exists $\eta\in\R$ such that for any $\alpha>0$, $s\in \R$ and $\bar Z=(\bar U,\bar V)\in E$ then
\begin{equation}\label{motonic}
\abs{\bar Z}_E\leq \abs{\bar Z-\alpha\big(e^{-Bs}\bar F_E(e^{Bs}\bar Z)-\eta\bar Z\big)}_E,\text{where}\;B=\sum_{k=1}^nB_k.
\end{equation}
It follows from \cite[Proof of Theorem 2]{Prato85},
this condition is needed to guarantee the strict
solution of abstract problem (6) in \cite{Prato85} does not explode in finite time.
Although reference \cite{Prato85} only focused on the existence and uniqueness of the strict solution of equation (6),
substituting coefficients in the system we are considering into (6) of \cite{Prato85}, a similar proof as in Lemma \ref{positivity} leads to the positivity for the solution of (6).
Hence, we need only verify the condition \eqref{motonic} for $\bar Z=(\bar U,\bar V)\in E$ with $\bar U(x),\bar V(x)\geq 0$ almost everywhere in $\0$ or $\bar U(x),\bar V(x)\geq 0\;\forall x\in \bar \0.$ As a consequence, by choosing $\eta\geq \max\big\{1+\abs{a_1}^*,1+\big|\frac{c_2}{m_2}\big|^*\big\}$, \eqref{motonic} is clearly satisfied. Therefore, the existence and uniqueness of strong solution are obtained by applying the results in \cite{Prato85}.
It is similar to Lemma \ref{lembounded}, we have for any finite $T>0$, $p\geq 1$,
$$(\bar U_n(t),\bar V_n(t))\in L^p(\Omega, C([0,T],E)).$$
\end{proof}

\begin{prop}\label{convergence}
For any $t\geq 0$, $p\geq 2$ and non-negative initial data $(U_0,V_0)\in D(A_E)$, we have
\begin{equation}\label{convergeS}
\lim_{n\to\infty}\E\abs{U(t)-\bar U_n(t)}_{L^2(\0,\R)}^p=0,
\end{equation}
and
\begin{equation}\label{convergeI}
\lim_{n\to\infty}\E\abs{V(t)-\bar V_n(t)}_{L^2(\0,\R)}^p=0,
\end{equation}
where $Z(t)=(U(t),V(t))$ is the mild solution of \eqref{so-vec} and $\bar Z_n(t)=(\bar U_n(t),\bar V_n(t))$ is the strong solution of \eqref{StrongSolution}.
\end{prop}

\begin{proof}
%It is obvious that we can assume $p>2$.
In this proof, the letter $c$ still denotes positive constants whose values may change in
different occurrences. We will write the dependence of constant on parameters explicitly if it is essential.
It is similar to Lemma \ref{lembounded}, we can obtain
$$\E\sup_{s\in[0,t]}\abs{Z_n(s)}_{L^\infty(\0,\R^2)}^p\leq c_{p,Z_0}(t)\ \text{for some constant }\ c_{p,Z_0}(t),\ \text{ is independent of }\ n.$$
Therefore, as in  part (ii) of Remark \ref{rm2}, we also obtain a similar convergence for the solution $Z_n(t)$ of \eqref{StrongSolution} and their truncated solutions. Moreover, this convergence is uniform with respect to $n$.
So, without loss of the generality, we can assume the non-linear term $F$ is Lipschitz continuous in this proof since we can approximate solutions of \eqref{so-vec} and \eqref{StrongSolution} by their truncated solutions.

 First, we still assume that each
 %$\forall\;1\leq
 $k\in \N$, $e_k\in C^2(\bar\0,\R).$  Because a strong solution is also a mild one, we have
\begin{equation}\label{Z_n-vec}
\displaystyle \bar Z_n(t)=e^{tA}Z_0+\int_0^t e^{(t-s)A}F(\bar Z_n(s))ds+W_{\bar Z_n}(t), \;Z_0=(U_0,V_0),
\end{equation}
where $W_{\bar Z_n}(t)=(W_{\bar U_n}(t), W_{\bar V_n}(t))$ and
$$W_{\bar U_n}(t)=\sum_{k=1}^n\lambda_{k,1}\int_0^t e^{(t-s)A_1}\bar U_n(s)dB_{k,1}(s),$$
$$W_{\bar V_n}(t)=\sum_{k=1}^n\lambda_{k,2}\int_0^t e^{(t-s)A_2}\bar V_n(s)dB_{k,2}(s).$$
By the same argument as in the processing of getting \eqref{estimateW}, we obtain
\begin{equation}\label{estimateWZ_n}
\abs{W_Z-W_{\bar Z_n}}_{L_{t,p}}\leq c_p(t)\int_0^t\abs{Z-\bar Z_n}_{L_{s,p}}ds+c_p(t)\sum_{k=n}^\infty(\lambda_{k,1}+\lambda_{k,2})\abs{Z}_{L_{t,p}}.
\end{equation}
Subtracting \eqref{so-vec} side-by-side from \eqref{Z_n-vec} and applying \eqref{bounded}, \eqref{estimateWZ_n} allows us to get
\begin{equation*}
\abs{Z-\bar Z_n}_{L_{t,p}}\leq c_{p,Z_0}(t)\Big(\sum_{k=n}^\infty(\lambda_{k,1}+\lambda_{k,2})+\int_0^t\abs{Z-\bar Z_n}_{L_{s,p}}ds\Big),
\end{equation*}
for some constant $c_{p,Z_0}(t)$ independent of $n$. Hence, it follows from  Gronwall's inequality that
\begin{equation}\label{Z-Z_n}
\abs{Z-\bar Z_n}_{L_{t,p}}\leq c_{p,Z_0}(t)\Big[\sum_{k=n}^\infty(\lambda_{k,1}+\lambda_{k,2})\Big].
\end{equation}
By
\eqref{nuclearcondition}, it is seen that
\begin{equation}\label{sumtoinfty}
\lim_{n\to\infty}\sum_{k=n}^\infty(\lambda_{k,1}+\lambda_{k,2})=0.
\end{equation}
Thus, we obtain from \eqref{Z-Z_n} and \eqref{sumtoinfty} that
$$\lim_{n\to\infty}\abs{Z-\bar Z_n}_{L_{t,p}}=0.$$
As a consequence, for all $t\geq 0, p\geq 2$
$$\lim_{n\to\infty}\E\abs{Z(t)-\bar Z_n(t)}^p_{L^2(\0,\R^2)}=0.$$
Now, as the above proof, by the fact $C^{\infty}(\bar\0,\R)$ is dense in $L^2(\0,\R)$, we can  remove
the condition $e_k\in C^2(\bar\0,\R)$. To be more detailed, we will first approximate the mild solution of \eqref{so-vec} by a sequence of mild solutions of \eqref{StrongSolution} without the condition $e_k\in C^2(\bar\0,\R)$ and then these solutions are approximated by the strong solutions of \eqref{StrongSolution} with condition $e_k\in C^2(\bar\0,\R)\;\forall 1\leq k\leq n$. Therefore, from now, to simplify the notation, we will approximate directly the mild solution of \eqref{so-vec} by the strong solutions of \eqref{StrongSolution}. Equivalently, without loss of the generality, we may assume that $e_k\in C^2(\bar \0,\R)\;\forall k=1,2,...$
as far as the approximation is concerned.
\end{proof}

\begin{rem}{\rm Combining Remark \ref{rm2} and the above proof, the convergence in Proposition \ref{convergence} still holds in the space $L^p(\Omega; C([0,T],L^\infty(\0,\R^2)))$. In more details, by the same arguments, it is possible to obtain that for any finite $T>0$, $p\geq 1$
$$\lim_{n\to\infty}\E \sup_{s\in [0,T]}\abs{Z(s)-\bar Z_n(s)}_{L^\infty(\0,\R^2)}^p=0.$$
}
\end{rem}

Now, for each $m\in \N, m>\big|\frac1{U_0}\big|_{L^{\infty}(\0,\R)}, \big|\frac1{V_0}\big|_{L^{\infty}(\0,\R)}$ (if they are finite), let
$$\tau_m^n=\inf\Big\{t\geq 0: \text{there exists}\;1\leq \bar p< \infty\;\text{such that}\;\int_\0\dfrac 1{\bar U_n^p(t,x)}dx\geq m^{p^2} \;\forall p\geq\bar p\Big\},$$
$$\eta_m^n=\inf\Big\{t\geq 0: \text{there exists}\;1\leq \bar p< \infty\;\text{such that}\;\int_\0\dfrac 1{\bar V_n^p(t,x)}dx\geq m^{p^2} \;\forall p\geq\bar p\Big\}.$$
It is easy to see that for any fix $n\in\N$, the sequences $\tau^n_m$ and $\eta^n_m$ are increasing in $m.$ Hence, we can define
$$\tau_\infty^n:=\lim_{m\to\infty}\tau_m^n\;, \ \eta_\infty^n:=\lim_{m\to\infty}\eta_m^n.$$

\begin{lem}
If $\big|\frac 1{U_0}\big|_{L^{\infty}(\0,\R)}<\infty$ and $\big|\frac 1{V_0}\big|_{L^{\infty}(\0,\R)}<\infty$, then for all $n\in\N$, $\tau_\infty^n=\eta_\infty^n=\infty$\a.s
\end{lem}

\begin{proof}
First, we prove that $\forall n\in\N,\;\tau_\infty^n=\infty\a.s$
Indeed, if this statement is false then there exist $n_0\in \N$ and two constants $T_0>0$ and $\eps_0\in(0,1)$ such that
$$\PP\{\tau_\infty^{n_0}\leq T_0\}>\eps_0.$$
Therefore, there is an integer $m_0$ such that $\PP(\Omega_0^m)\geq \eps_0, \;\forall m\geq m_0$, where
\begin{equation}\label{tau_m}
\Omega_0^m:=\{\tau_m^{n_0}\leq T_0\}.
\end{equation}
Using It\^o's Lemma (\cite[Theorem 3.8]{Curtain}) and by direct calculations, we have
\begin{equation}\label{1/s^2}
\begin{aligned}
&\!\!\!\int_\0 \dfrac 1{\bar U^p_n(t\wedge\tau_m^n,x)}dx \\ & \ =\int_\0\dfrac 1{U_0^p(x)}dx+\int_0^{t\wedge\tau_m^n}\int_\0\dfrac {-p}{\bar U^{p+1}_n(s,x)}\\
&\qquad\times \Big(d_1\Delta \bar U_n(s,x)+\bar U_n(s,x)\big(a_1(x)-b_1(x)\bar U_n(s,x)\big)
\\&\qquad-\dfrac{c_1(x)\bar U_n(s,x)\bar V_n(s,x)}{m_1(x)+m_2(x)\bar U_n(s,x)+m_3(x)\bar V_n(s,x)}\Big)dxds
\\&\qquad+\dfrac 12 \int_0^{t\wedge\tau_m^n}\sum_{k=1}^n\int_\0 \dfrac{p(p+1)\lambda_{k,1}e^2_k(x)\bar U_n^2(s,x)}{\bar U_n^{p+2}(s,x)}dxds\\
&\qquad +\sum_{k=1}^n\int_0^{t\wedge\tau_m^n}\Big[\sqrt{\lambda_{k,1}}\int_\0 \dfrac{-pe_k(x)\bar U_n(s,x)}{\bar U_n^{p+1}(s,x)}dx\Big]dB_{k,1}(s)
\\&\leq \int_\0\dfrac 1{U_0^p(x)}dx+\int_0^{t\wedge\tau_m^n}\int_\0\dfrac{-pd_1\Delta\bar U_n(s,x)}{\bar U_n^{p+1}(s,x)}dxds
\\&\qquad+\int_0^{t\wedge\tau_m^n}\int_\0 \dfrac{p}{\bar U_n^{p}(s,x)}\Big(K_1+pK_2+\abs{b_1}^*\bar U_n(s,x)\Big)dxds
\\&\qquad+\sum_{k=1}^n\int_0^{t\wedge\tau_m^n}\Big[\sqrt{\lambda_{k,1}}\int_\0 \dfrac{-pe_k(x)}{\bar U_n^p(s,x)}dx\Big]dB_{k,1}(s)
\\&\leq \int_\0\dfrac 1{U_0^p(x)}dx+\int_0^{t\wedge\tau_m^n} \Big(K_3(p)+\int_\0 \dfrac{K_3(p)}{\bar U_n^{p}(s,x)}dx\Big)ds
\\&\qquad+\sum_{k=1}^n\int_0^{t\wedge\tau_m^n}\Big[\sqrt{\lambda_{k,1}}\int_\0 \dfrac{-pe_k(x)}{\bar U_n^p(s,x)}dx\Big]dB_{k,1}(s),\;\forall p\geq 1, t\geq 0,n\in\N,
\end{aligned}
\end{equation}
where $K_1=\big|\frac{c_1}{m_3}\big|^*+\frac{\lambda_1C_0^2}2$, $K_2=\frac{\lambda_1C_0^2}2$ and $K_3(p)=p\big(K_1+pK_2+\abs{b_1}^*\big).$
In the above, we used the following facts
$$\int_\0\dfrac{-pd_1\Delta\bar U_n(s,x)}{\bar U_n^{p+1}(s,x)}dx=-p(p+1)d_1\int_\0\dfrac{\abs{\nabla \bar U_n(s,x)}^2}{\bar U_n^{p+2}(s,x)}dx\leq 0 \ \hbox{ a.s.,}$$
and
\begin{equation*}
\begin{aligned}
\int_\0& \dfrac{p}{\bar U_n^{p}(s,x)}\Big(K_1+pK_2+\abs{b_1}^*\bar U_n(s,x)\Big)dx
\\&=\int_\0 \dfrac{p}{\bar U_n^{p}(s,x)}\Big(K_1+pK_2+\abs{b_1}^*\bar U_n(s,x)\Big)\1_{\{U(s,x)\geq 1\}}dx
\\&\quad+\int_\0 \dfrac{p}{\bar U_n^{p}(s,x)}\Big(K_1+pK_2+\abs{b_1}^*\bar U_n(s,x)\Big)\1_{\{U(s,x)\leq 1\}}dx
\\&\leq p\big(K_1+pK_2+\abs{b_1}^*\big)+\int_\0 \dfrac{p\big(K_1+pK_2+\abs{b_1}^*\big)}{\bar U_n^{p}(s,x)}dx\a.s
\end{aligned}
\end{equation*}
Hence, \eqref{1/s^2} leads to
\begin{equation*}
\E\int_\0 \dfrac 1{\bar U^p_n(t\wedge\tau_m^n,x)}dx\leq \int_\0\dfrac 1{U_0^p(x)}dx+tK_3(p)+K_3(p)\int_0^{t} \E\int_\0 \dfrac{1}{\bar U_n^{p}(s\wedge\tau_m^n,x)}dxds.
\end{equation*}
Thus, Gronwall's inequality implies that
\begin{equation*}
\E\int_\0 \dfrac 1{\bar U^p_n(t\wedge\tau_m^n,x)}dx\leq \Big[\int_\0\dfrac 1{U_0^p(x)}dx+tK_3(p)\Big]e^{tK_3(p)}.
\end{equation*}
Therefore, for each fixed $t\geq 0$ and $\forall n\in \N$,
\begin{equation*}\label{e1/s^p}
\begin{aligned}
\sup_{p\geq 1}\Big[\E&\int_\0 \dfrac 1{\bar U^{p}_n(t\wedge\tau_m^n,x)}dx\Big]^{1/p^2}
\\&\leq \sup_{p\geq 1} \Big[\int_\0\dfrac 1{U_0^p(x)}dx+\big(pK_1+p\abs{b_1}^*+p^2K_2\big)
t\Big]^{1/p^2}e^{t(\frac{K_1+\abs{b_1}^*}p+K_2)}\\
&:=M(t)<\infty.
\end{aligned}
\end{equation*}
In particular,
\begin{equation}\label{47}
 \sup_{p\geq 1}\Big[\E\int_\0 \dfrac 1{\bar U^p_n(T_0\wedge\tau_m^n,x)}dx\Big]^{1/p^2}\leq M(T_0)\;\forall n\in \N.
\end{equation}
On the other hand, for all $m\geq m_0$,
\begin{equation}\label{48}
\begin{aligned}
\sup_{p\geq 1} \Big[\E\int_\0 \dfrac 1{\bar U^p_{n_0}(T_0\wedge\tau_m^{n_0},x)}dx\Big]^{1/p^2}&\geq\sup_{p\geq 1}\Big[\E\1_{\Omega_0^m} \int_\0 \dfrac 1{\bar U^p_{n_0}(T_0\wedge\tau_m^{n_0},x)}dx\Big]^{1/p^2}
\\&\geq m\eps_0.
\end{aligned}
\end{equation}
We deduce from \eqref{47} and \eqref{48} that
$$M(T_0)\geq m\eps_0\quad\forall m\geq m_0.$$
This is a contradiction when $m\to\infty$. Therefore
\begin{equation*}
\tau_\infty^n=\infty\a.s\;\forall n\in\N.
\end{equation*}
Similarly, we obtain that
$$\eta_\infty^n=\infty\a.s\;\forall n\in\N.$$
\end{proof}

To proceed, we introduce following numbers
$$H_0=\Big|a_1-\frac{c_1}{m_3}\Big|_*-\dfrac{3\lambda_1C_0^2}{2}:=\inf_{x\in\bar \0}\Big[a_1(x)-\dfrac{c_1(x)}{m_3(x)}\Big]-\dfrac{3\lambda_1C_0^2}{2},$$
$$R_0=-\abs{a_2}_1-\dfrac {\lambda_2}2+\dfrac{1}{\big|\frac{m_2}{c_2}\big|^*+\min\big\{\frac 1{H_0}\abs{b_1}_1\big|\frac{m_1}{c_2}\big|^*\; ;\;\frac {1}{H_0}\abs{b_1}_1^{1/2}(\abs{b_1}^*)^{1/2}\big|\frac{m_1}{c_2}\big|_2\big\}}.$$

\begin{thm} The following results hold.
\begin{itemize}
\item[{\rm(i)}] For any initial values $0\leq U_0,V_0 \in L^\infty(\0,\R)$. If
$$\Big|a_2-\frac{c_2}{m_2}\Big|_*:=\inf_{x\in\bar \0}\Big\{a_2(x)-\frac{c_2(x)}{m_2(x)}\Big\}>0,$$
then $V(t)$ is extinct.
\item[{\rm(ii)}]If $H_0,R_0>0$ and non-negative initial values $(U_0,V_0)\in E$ satisfying
$$\Big|\frac 1{U_0}\Big|_{L^{\infty}(\0,\R)}<\infty\;,\;\Big|\frac 1{V_0}\Big|_{L^{\infty}(\0,\R)}<\infty,$$ the individuals $U(t)$ and $V(t)$ are permanent.
\end{itemize}
\end{thm}

\begin{proof}
The proof for the first part is similar to \cite[ Proof of Theorem 4.1]{NY-18}. It follows \eqref{so} and properties of stochastic integral (\cite[Proposition 4.15]{Prato} and \cite[Proposition 2.9]{Curtain}) that $\forall t\geq 0$,
\begin{equation}
\begin{aligned}
0\leq \E\int_\0V(t,x)dx&\leq \int_\0 V_0(x)dx +\int_{0}^t \E\int_\0 \Big(-a_2(x)V(s,x)+\dfrac{c_2(x)V(s,x)}{m_2(x)}\Big)dxds
\\& \leq \int_\0 V_0(x)dx-\Big|a_2-\frac{c_2}{m_2}\Big|_*\int_0^t\E\int_\0 V(s,x)dxds.
\end{aligned}
\end{equation}
As a consequence, we have
$$\E \int_\0V(t,x)dx\leq -\big|a_2-\frac{c_2}{m_2}\big|_*\int_s^t \left(\int_\0V(s,x)dx\right)ds,\;\forall 0\leq s\leq t.$$
Hence, it leads to that $\lim_{t\to\infty}\E\int_\0 V(t,x)dx=0$ with exponential rate or the class $V(t)$ is extinct.
Now, we move to the second part. Since the density of $D(A_E)$ in $E$ and the continuous dependence on initial values of the solution, we can assume that $(U_0,V_0)\in D(A_E)$.
By It\^o's formula (\cite[Theorem 3.8]{Curtain}) and by a similar argument as in the processing of getting \eqref{1/s^2}, we have
\begin{equation*}
\begin{aligned}
&\!\!\! e^{H_0(t\wedge\tau_m^n)}
 \int_\0\dfrac{1}{\bar U_n(t\wedge\tau_m^n,x)}dx\\
 & \leq\int_\0 \dfrac 1{U_0(x)}dx
%\\&\quad
+\int_0^{t\wedge \tau_m^n}e^{H_0s}\Big(\int_\0\dfrac{\frac{c_1(x)}{m_3(x)}+\lambda_1C_0^2+H_0-a_1(x)}{\bar U_n(s,x)}dx+\int_\0 b_1(x)dx\Big)ds
\\&\quad+\sum_{k=1}^n\int_0^{t\wedge \tau_m^n}e^{H_0s}\Big[\int_\0\dfrac{-\sqrt{\lambda_{k,1}}e_k(x)}{\bar U_n(s,x)}dx\Big]dB_{k,1}(s)
\\&\leq \int_\0 \dfrac 1{U_0(x)}dx+\int_0^{t\wedge \tau_m^n}\abs{b_1}_1e^{H_0s}ds
\\&\quad
+\sum_{k=1}^n\int_0^{t\wedge \tau_m^n}e^{H_0s}\Big[\int_\0\dfrac{-\sqrt{\lambda_{k,1}}e_k(x)}{\bar U_n(s,x)}dx\Big]dB_{k,1}(s).
\end{aligned}
\end{equation*}
Therefore, taking  expectations on both sides and letting $m\to\infty$, we obtain
\begin{equation*}
\E\int_\0\dfrac{1}{\bar U_n(t,x)}dx\leq e^{-H_0t}\int_\0 \dfrac 1{U_0(x)}dx+\abs{b_1}_1\dfrac{e^{H_0t}-1}{H_0e^{H_0t}},\quad\forall t\geq 0, \forall n\in \N.
\end{equation*}
As a consequence,
\begin{equation}\label{e1/u}
\limsup_{t\to\infty}\limsup_{n\to\infty}\E\int_\0\dfrac{1}{\bar U_n(t,x)}dx\leq \dfrac{\abs{b_1}_1}{H_0}.
\end{equation}
The convergence \eqref{convergeS}, the H\"older's inequality, and \eqref{e1/u} yield
$$\limsup_{t\to\infty}\E\int_\0U(t,x)dx=\limsup_{t\to\infty}\limsup_{n\to\infty}\E\int_\0 \bar U_n(t,x)dx\geq \dfrac {H_0}{\abs{b_1}_1}.$$
So, the individual $U(t)$ is permanent.
Similarly, we also obtain from It\^o's formula (\cite[Theorem 3.8]{Curtain}) that
\begin{equation*}
\begin{aligned}
\E\int_\0&\dfrac{1}{\bar U^2_n(t,x)}dx\leq e^{-2H_0t}\int_\0 \dfrac 1{U_0^2(x)}dx\\
&\quad
+e^{-2H_0t}\E\int_0^{t}e^{2H_0s}\Big(2\int_\0\dfrac{\frac{c_1(x)}{m_3(x)}
+\frac{3\lambda_1C_0^2}2+H_0-a_1(x)}{\bar U^2_n(s,x)}dx\\
&\quad +2\abs{b_1}^*\int_\0 \dfrac{1}{\bar U_n(s,x)}dx\Big)ds
\\&\leq e^{-2H_0t}\int_\0 \dfrac 1{U_0^2(x)}dx
%\\& \quad 
+2\abs{b_1}^* e^{-2H_0t}\int_0^{t}e^{2H_0s}\E\int_\0 \dfrac{1}{\bar U_n(s,x)}dxds,
\end{aligned}
\end{equation*}
which implies that
\begin{equation}\label{1/u^2}
\begin{aligned}
\limsup_{n\to\infty}\E\int_\0\dfrac{1}{\bar U^2_n(t,x)}dx&\leq e^{-2H_0t}\int_\0 \dfrac 1{U_0^2(x)}dx
\\&\quad+2\abs{b_1}^* e^{-2H_0t}\int_0^{t}e^{2H_0s}\Big[\limsup_{n\to\infty}\E\int_\0 \dfrac{1}{\bar U_n(s,x)}dx\Big]ds.
\end{aligned}
\end{equation}
Applying \eqref{e1/u} to \eqref{1/u^2}, we obtain
\begin{equation}\label{e1/u^2}
\limsup_{t\to\infty}\limsup_{n\to\infty}\E\int_\0\dfrac{1}{\bar U^2_n(t,x)}dx\leq \dfrac{\abs{b_1}_1\abs{b_1}^*}{H_0^2}.
\end{equation}
Since the definition of $R_0$ and $R_0>0$,
there exists a $\delta>0$ such that
\begin{equation}\label{R_1/2}\begin{array}{ll}
&\!\!\!\disp
-\abs{a_2}_1-\dfrac {\lambda_2}2-\delta\\
&\qquad +\dfrac{1}{\big|\frac{m_2}{c_2}\big|^*+\min\big\{\frac 1{H_0}\abs{b_1}_1\big|\frac{m_1}{c_2}\big|^*+\delta\; ;\;\frac {1}{H_0}\abs{b_1}_1^{1/2}(\abs{b_1}^*)^{1/2}\big|
\frac{m_1}{c_2}\big|_2+\delta\big\}+\delta}\\
& \quad \geq \dfrac{R_0}{2}.\end{array}
\end{equation}
Put $\hat\delta=\min\Big\{\dfrac{\delta^2}{4\abs{b_2}_2^2}\;;\;\dfrac{H_0^2\delta^2}{4(\big|\frac{m_3}{c_2}\big|^*)^2(\abs{b_1}^*)^2}\Big\}$. By a contradiction argument, we assume that
\begin{equation}\label{contradiction}
\limsup_{t\to\infty}\E\int_\0 V^2(t,x)dx<\hat\delta,
\end{equation}
which means
\begin{equation}\label{contraction}
\limsup_{t\to\infty}\lim_{n\to\infty}\int_\0 \bar V_n^2(t,x)dx<\hat\delta.
\end{equation}
It follows from \eqref{contraction} that
\begin{equation*}
\begin{aligned}
\limsup_{t\to\infty}\limsup_{n\to\infty}\E\int_\0 b_2(x)\bar V_n(t,x)dx
&\leq\abs{b_2}_2 \limsup_{t\to\infty}\limsup_{n\to\infty}\Big(\E\int_\0\bar V_n^2(t,x)dx\Big)^{1/2}
\\&\leq \abs{b_2}_2\hat\delta^{1/2}\leq \dfrac{\delta}2.
\end{aligned}
\end{equation*}
It implies that there exists a $T_1$ such that
\begin{equation}\label{T_1}
\limsup_{n\to\infty}\E\int_\0 b_2(x)\bar V_n(t,x)dx\leq \delta\quad\forall t\geq T_1.
\end{equation}
On the other hand, it follows from H\"older's inequality, \eqref{e1/u^2}, and \eqref{contraction} that
\begin{equation*}
\begin{aligned}
\limsup_{t\to\infty}\limsup_{n\to\infty}&\;\E\int_\0\dfrac{\bar V_n(t,x)}{\bar U_n(t,x)}dx
\\&\leq \limsup_{t\to\infty}\limsup_{n\to\infty}\Big(\E\int_\0\bar V_n^2(t,x)dx\Big)^{1/2}\\
&\quad \times
\limsup_{t\to\infty}\limsup_{n\to\infty}\Big(\E\int_\0\dfrac 1{\bar U_n^2(t,x)}dx\Big)^{1/2}
\\&\leq \big(\frac 1{H_0^2}\hat\delta\abs{b_1}_1\abs{b_1}^*\big)^{1/2}\leq\dfrac{\delta}{2\big|\frac{m_3}{c_2}\big|^*}.
\end{aligned}
\end{equation*}
Therefore, there exists $T_2$ such that
\begin{equation}\label{T_2}
\limsup_{n\to\infty}\;\Big|\frac{m_3}{c_2}\Big|^*\E\int_\0\dfrac{\bar V_n(t,x)}{\bar U_n(t,x)}dx\leq \delta\quad\forall t\geq T_2.
\end{equation}
Moreover, it is clear that
\begin{equation}\label{m_1/c_2U}\begin{array}{ll}
\disp \E\int_\0\frac{m_1(x)}{c_2(x)\bar U_n(t,x)}dx\leq \min\Big\{&\!\!\!\Big|\dfrac{m_1}{c_2}\Big|^*\E\int_\0\dfrac 1{\bar U_n(t,x)}dx,\\
 & \quad \Big|\dfrac{m_1}{c_2}\Big|_2\Big(\E\int_\0\dfrac 1{\bar U_n^2(t,x)}dx\Big)^{1/2}\Big\}.\end{array}
\end{equation}
Combining \eqref{e1/u},\eqref{e1/u^2} and \eqref{m_1/c_2U} obtains
\begin{equation*}
\begin{aligned}
\limsup_{t\to\infty}&\limsup_{n\to\infty}\E\int_\0\frac{m_1(x)}{c_2(x)\bar U_n(t,x)}dx
\\&\leq \min\Big\{\frac 1{H_0}\abs{b_1}_1\Big|\dfrac{m_1}{c_2}\Big|^*,\; \;\frac {1}{H_0}\abs{b_1}_1^{1/2}(\abs{b_1}^*)^{1/2}\Big|\dfrac{m_1}{c_2}\Big|_2\Big\}.
\end{aligned}
\end{equation*}
Thus, there exists $T_3$ such that for all $t\geq T_3$
\begin{equation}\label{T_3}\barray{ll}
&\!\!\!\disp
\limsup_{n\to\infty}\E\int_\0\frac{m_1(x)}{c_2(x)\bar U_n(t,x)}dx\leq  \min\Big\{\frac 1{H_0}\abs{b_1}_1\Big|\dfrac{m_1}{c_2}\Big|^*+\delta, \\ 
&\quad\disp \frac {1}{H_0}\abs{b_1}_1^{1/2}(\abs{b_1}^*)^{1/2}\Big|\dfrac{m_1}{c_2}\Big|_2+\delta\Big\}.
\end{array}\end{equation}
Now, by It\^o's formula (\cite[Theorem 3.8]{Curtain}) again we have
\begin{equation}
\begin{aligned}
&\int_\0\ln\bar V_n(t\wedge \eta_m^n,x)dx
\\&\geq \int_\0 \ln V_0(x)dx+\int_0^{t\wedge\eta_m^n}\int_\0\dfrac{d_2\abs{\nabla \bar V_n(s,x)}^2}{\bar V^2_n(s,x)}dxds-\dfrac {\lambda_2}2(t\wedge \eta_m^n)
\\&\quad+\int_0^{t\wedge \eta_m^n}\int_\0\Big(-a_2(x)-b_2(x)\bar V_n(s,x)\\
& \quad +\dfrac{c_2(x)\bar U_n(s,x)}{m_1(x)+m_2(x)\bar U_n(s,x)+m_3(x)\bar V_n(s,x)}\Big)dxds
\\&\quad +\int_0^{t\wedge\eta_m^n}\sum_{k=1}^n\sqrt{\lambda_{k,2}}\Big(\int_\0 e_k(x)dx\Big)dB_{k,2}(s).
\end{aligned}
\end{equation}
Taking expectation on both sides and letting $m\to\infty$ imply
\begin{equation}\label{elnV_n}
\begin{aligned}
\E\int_\0&\ln \bar V_n(t,x)dx\geq \int_\0 \ln V_0(x)dx-\big(\abs{a_2}_1+\dfrac {\lambda_2}2\big)t
\\&+\int_0^{t}\E\int_\0\Big(-b_2(x)\bar V_n(s,x)\\
& \quad +\dfrac{c_2(x)\bar U_n(s,x)}{m_1(x)+m_2(x)\bar U_n(s,x)+m_3(x)\bar V_n(s,x)}\Big)dxds.
\end{aligned}
\end{equation}
By Jensen's inequality, we obtain
\begin{equation}\label{Jensen}
\begin{aligned}
\E\int_\0&\dfrac{c_2(x)\bar U_n(s,x)}{m_1(x)+m_2(x)\bar U_n(s,x)+m_3(x)\bar V_n(s,x)}dx
\\&\geq \E\int_\0\dfrac{1}{\big|\frac{m_2}{c_2}\big|^*+\frac{m_1(x)}{c_2(x)\bar U_n(s,x)}+\big|\frac{m_3}{c_2}\big|^*\frac{\bar V_n(s,x)}{\bar U_n(s,x)}}dx
\\&\geq \dfrac{1}{\big|\frac{m_2}{c_2}\big|^*+\E\int_\0\frac{m_1(x)}{c_2(x)\bar U_n(s,x)}dx+\big|\frac{m_3}{c_2}\big|^*\E\int_\0\frac{\bar V_n(s,x)}{\bar U_n(s,x)}dx}.
\end{aligned}
\end{equation}
Applying \eqref{Jensen} to \eqref{elnV_n}, we have
\begin{equation}\label{elnV}
\begin{aligned}
&\limsup_{n\to\infty}\E\int_\0\ln \bar V_n(t,x)dx\geq \int_\0 \ln V_0(x)dx-\big(\abs{a_2}_1+\dfrac {\lambda_2}2\big)t
\\&\;\;-\int_0^{t}\Big[\limsup_{n\to\infty}\E\int_\0b_2(x)\bar V_n(s,x)dx\Big]ds
\\&\;\;+ \int_0^t\dfrac{1}{\Big|\frac{m_2}{c_2}\Big|^*+\displaystyle\limsup_{n\to\infty}\E\int_\0\frac{m_1(x)}{c_2(x)\bar U_n(s,x)}dx+\Big|\frac{m_3}{c_2}\Big|^*\displaystyle\limsup_{n\to\infty}\E\int_\0\frac{\bar V_n(s,x)}{\bar U_n(s,x)}dx}ds.
\end{aligned}
\end{equation}
Let $T=\max\{T_1,T_2,T_3\}.$  We obtain from \eqref{elnV},\eqref{T_1}, \eqref{T_2}, \eqref{T_3}, and \eqref{R_1/2} that $\forall t\geq T$
\begin{equation}
\begin{aligned}
\limsup_{n\to\infty}\;&\E\int_\0\ln \bar V_n(t,x)dx
\\&\geq \int_\0 \ln V_0(x)dx-\big(\abs{a_2}_1+\dfrac {\lambda_2}2\big)T-\int_0^{T}\Big[\limsup_{n\to\infty}\E\int_\0b_2(x)\bar V_n(s,x)dx\Big]ds
%\\&+ \int_0^T\dfrac{1}{\big|\frac{m_2}{c_2}\big|_E+\displaystyle\limsup_{n\to\infty}\E\int_\0\frac{m_1(x)}{c_2(x)\bar U_n(s,x)}dx+\Big|\frac{m_3}{c_2}\Big|_E\displaystyle\limsup_{n\to\infty}\E\int_\0\frac{\bar V_n(s,x)}{\bar U_n(s,x)}dx}ds
\\&+\dfrac{R_0(t-T)}2.
\end{aligned}
\end{equation}
Since $R_0>0$, we have
$$
\limsup_{t\to\infty}\limsup_{n\to\infty}\E\int_\0\ln \bar V_n(t,x)dx=\infty.
$$
Therefore, it follows from the convergence \eqref{convergeI} and Jensen's inequality that
\begin{equation}\label{limsupelnv}
\begin{aligned}
\limsup_{t\to\infty}\ln \E\int_\0V(t,x)dx&=\limsup_{t\to\infty}\ln \big[\limsup_{n\to\infty}\E\int_\0\bar V_n(t,x)dx\big]
\\&\geq \limsup_{t\to\infty}\limsup_{n\to\infty}\E\int_\0\ln \bar V_n(t,x)dx=\infty.
\end{aligned}
\end{equation}
However, combining \eqref{contradiction} and \eqref{limsupelnv} leads to a contradiction. This implies
$$\limsup_{t\to\infty}\int_\0 V^2(t,x)dx\geq\hat\delta.
$$
Thus, the individual $V(t)$ is permanent.
\end{proof}

\section{An Example}
\label{example}
In this section,
we consider an example when the processes driving noise processes in
equation \eqref{eq} are standard Brownian motions and the coefficients are independent of space, as following
\begin{equation}\label{eq-1}
\begin{cases}
d
 U(t,x)\!\!=\!\!\Big[d_1\Delta U(t,x)+U(t,x)\big(a_1-b_1U(t,x)\big)\!-\!
 \dfrac{c_1U(t,x)V(t,x)}{m_1+m_2U(t,x)+ m_3V(t,x)}\Big]dt
\\ \quad \hfill +\sigma_1U(t,x)dB_1(t)\quad \text{in}\ \R^+\times\0, \\
d
 V(t,x)\!\!=\!\!\Big[d_2\Delta V(t,x)\!-\!V(t,x)\big(a_2+b_2V(t,x)\big)\!+\!\1\dfrac{c_2 U(t,x)V(t,x)}{m_1+m_2U(t,x)+ m_3V(t,x)}\Big]dt
\\
\quad \hfill+\sigma_2 V(t,x)dB_2(t)\quad\text{in} \ \R^+\times\0,\\[1ex]
\partial_{\nu}U(t,x)=\partial_{\nu}V(t,x)=0\quad
\quad\quad\quad\quad\quad\quad\text{on} \;\;\;\;\R^+\times\partial\0,\\
U(x,0)=U_0(x),V(x,0)=V_0(x)\quad\quad\quad\quad\text{in} \;\;\;\;\0,
\end{cases}
\end{equation}
where $a_i,b_i,c_i,m_i$ are positive constants, and $B_1(t)$, $B_2(t)$ are independent standard Brownian motions.
As we obtained above, for any initial values $0\leq U_0,V_0\in L^\infty(\0,\R)$,
\eqref{eq-1} has unique solution $U(t,x),V(t,x)\geq 0.$ Moreover, the long-time behavior of the system is shown as the following theorem.

\begin{thm}\label{longtimeexample}
Let $U(t,x),V(t,x)$ be the mild solution of equation \eqref{eq-1}.
\begin{itemize}
\item[{\rm(i)}] For any non-negative initial values $U_0,V_0\in L^\infty(\0,\R)$, if $a_2>\dfrac{c_2}{m_2}$, then the individuals $V(t)$ is extinct.
\item[{\rm(ii)}] Assume that non-negative initial values $(U_0,V_0)\in C(\bar\0,\R^2)$ satisfy $$\Big|\frac1{U_0}\Big|_{L^{\infty}(\0,\R)} <\infty
  \hbox{ and } \ \Big|\frac1{V_0}\Big|_{L^{\infty}(\0,\R)} <\infty.$$
If $d:=a_1-\dfrac{c_1}{m_3}-\dfrac {3\sigma_1^2}2>0$, and $\dfrac{dc_2}{m_2d+b_1m_1}>a_2+\dfrac{\sigma_2^2}{2}$, then the classes $U(t),V(t)$ are permanent.
\end{itemize}
\end{thm}

\end{document}